\documentclass[a4, 11pt] {amsart} 
\usepackage{amssymb,times, amscd,amsmath,amsthm, xypic}
\usepackage{geometry}
\author{Shoji Yokura$^{(*)}$}

\address{Department of Mathematics and Computer Science,
Faculty of Science\\ University of Kagoshima,
21-35 Korimoto 1-Chome,
Kagoshima 890-0065, Japan\\
yokura@sci.kagoshima-u.ac.jp}

\date{}

\thanks{(*) Partially supported by JSPS KAKENHI Grant Numbers 23244008, 24340007, 24540085 }

\title [Hirzebruch $\chi_y$-genera of
complex algebraic fiber bundles]
{Hirzebruch $\chi_y$-genera of
complex algebraic fiber bundles\\
--- the multiplicativity of the signature modulo $4$ ---}

\begin{document} 
\numberwithin{equation}{section}
\newtheorem{thm}[equation]{Theorem}
\newtheorem{pro}[equation]{Proposition}
\newtheorem{prob}[equation]{Problem}
\newtheorem{qu}[equation]{Question}
\newtheorem{cor}[equation]{Corollary}
\newtheorem{con}[equation]{Conjecture}
\newtheorem{lem}[equation]{Lemma}
\theoremstyle{definition}
\newtheorem{ex}[equation]{Example}
\newtheorem{defn}[equation]{Definition}
\newtheorem{ob}[equation]{Observation}
\newtheorem{rem}[equation]{Remark}
\renewcommand{\rmdefault}{ptm}
\def\alp{\alpha}
\def\be{\beta}
\def\jeden{1\hskip-3.5pt1}
\def\om{\omega}
\def\bigstar{\mathbf{\star}}
\def\ep{\epsilon}
\def\vep{\varepsilon}
\def\Om{\Omega}
\def\la{\lambda}
\def\La{\Lambda}
\def\si{\sigma}
\def\Si{\Sigma}
\def\Cal{\mathcal}
\def\m {\mathcal}
\def\ga{\gamma}
\def\Ga{\Gamma}
\def\de{\delta}
\def\De{\Delta}
\def\bF{\mathbb{F}}
\def\bH{\mathbb H}
\def\bPH{\mathbb {PH}}
\def \bB{\mathbb B}
\def \bA{\mathbb A}
\def \bC{\mathbb C}
\def \bOB{\mathbb {OB}}
\def \bM{\mathbb M}
\def \bOM{\mathbb {OM}}
\def \mA{\mathcal A}
\def \mB{\mathcal B}
\def \mC{\mathcal C}
\def \mR{\mathcal R}
\def \mH{\mathcal H}
\def \mM{\mathcal M}
\def \mM{\mathcal M}
\def \mT{\mathcal {T}}
\def \mAB{\mathcal {AB}}
\def \bK{\mathbb K}
\def \bG{\mathbf G}
\def \bL{\mathbb L}
\def\bN{\mathbb N}
\def\bR{\mathbb R}
\def\bP{\mathbb P}
\def\bZ{\mathbb Z}
\def\bC{\mathbb  C}
\def \bQ{\mathbb Q}
\newcommand{\bb}[1]{\mbox{$\mathbb{#1}$}}

\def\op{\operatorname}

\maketitle

\begin{abstract} Let $E$ be a fiber $F$ bundle over a base $B$ such that $E, F$ and $B$ are smooth compact complex algebraic varieties. In this paper we give explicit formulae for the difference of the Hirzebruch $\chi_y$-genus $\chi_y(E) - \chi_y(F)\chi_y(B)$. As a byproduct of the formulae we obtain that the signature of such a fiber bundle is multiplicative mod $4$, i.e. the signature difference $\sigma(E) -\sigma(F)\sigma(B)$ is always divisible by $4$. In the case of $\op{dim}_{\mathbb C}E \leqq 4$ the $\chi_y$-genus difference $\chi_y(E) - \chi_y(F)\chi_y(B)$ can be concretely described only in terms of the signature difference $\sigma(E) -\sigma(F) \sigma(B)$ and/or the Todd genus difference $\tau(E) - \tau(F)\tau(B)$. Using this we can obtain that in order for
$\chi_y$ to be multiplicative $\chi_y(E) = \chi_y(F)\cdot \chi_y(B)$ for any such fiber bundle $y$ has to be $-1$, namely only the Euler-Poincar\'e characteristic is multiplicative for any such fiber bundle.
\end{abstract}


\section{Introduction}

The Hirzebruch $\chi_y$-genus has been introduced by F. Hirzebruch \cite{Hir} (also see \cite{HBJ}) in order to extend his famous Hirzebruch-Riemann-Roch theorem to a generalized one. If $y=-1, 0, 1$, then these $\chi_y$-genera are respectively $\chi_{-1}(X) =\chi(X)$ the \emph{Euler-Poincar\'e characteristic}, $\chi_0(X) = \tau(X)$ the \emph{Todd genus} and $\chi_1(X)=\sigma(X)$ the \emph{signature}, which are very important invariants in geometry and topology, and even in mathematical physics. In this sense the Hirzebruch $\chi_y$-genus unifies these three important characteristic numbers.

It is well-known that the Euler-Poincar\'e characteristic is multiplicative for \emph{any} topological fiber bundle, i.e. if $E$ is a fiber bundle with a fiber space $F$ and a base space $B$, then $\chi(E) = \chi(F)\chi(B)$ holds. As to the signature it is not the case. S. S. Chern, F. Hirzebruch and J.-P. Serre \cite{CHS} proved that the signature is multiplicative for a fiber bundle under a certain monodromy condition, i.e. if the fundamental group $\pi_1(B)$ of the base space $B$ acts trivially on the cohomology group $H^*(F; \mathbb R)$ of the fiber space $F$, for example, that is the case when $B$ is simply-connected. The signature is in general not multiplicative for fibre bundles, as shown by M. Atiyah \cite{At}, F. Hirzebruch \cite {Hir0} and K. Kodaira \cite{Ko}. These examples are all of real dimension $4$, which is the lowest possible one, with $\op{dim}_{\mathbb R} F = \op{dim}_{\mathbb R}B =2$. Which means that the signature of these examples are automatically non-zero, because $\sigma(F) = \sigma(B) =0$ by the definition of the signature (which is defined to be zero if the real dimension of the manifold is not divisible by $4$). H. Endo \cite{Endo}  and W. Meyer \cite{Mey}  further studied such surface bundles over surfaces, in which they show that the signature of such fiber bundles are always divisible by $4$, i.e. $\sigma(E) \equiv 0 \op{mod} 4.$

I. Hambleton, A. Korzeniewski and A. Ranicki  \cite{HKR}  have showed that for a $PL$ fibre bundle $F \hookrightarrow E \to B$ of closed, connected, compatibly oriented $PL$ manifolds the difference $\sigma(E) - \sigma(F)\sigma(B)$ is divisible by $4$, i.e. $\sigma(E) \equiv \sigma(F) \sigma(B) \op{mod} 4.$

The Hirzebruch $\chi_y$-genus is multiplicative, i.e. $\chi_y(X \times Y) = \chi_y(X)\chi_y(Y).$ In \cite{CMS} (also see \cite {CLMS1}, \cite{CLMS2}, \cite{MS}) S. Cappell, L. Maxim and J. Shaneson have shown that the Hirzebruch $\chi_y$-genus is also multiplicative for a fiber bundle, i.e., $\chi_y(E) = \chi_y(F)\chi_y(B)$, if the fundamental group $\pi_1(B)$ of the base space $B$ acts trivially on the cohomology group $H^*(F; \mathbb R)$ of the fiber space $F$, extending the above mentioned result of Chern-Hirzebruch-Serre \cite{CHS}. In fact they deal with the more general Hirzebruch $\chi_y$-genus of possibly singular varieties, which are defined to be the integral $\int_X {T_y}_*(X)$ of the motivic Hirzeburch class ${T_y}_*(X)$ introduced in \cite{BSY1} (cf. \cite{BSY2}). Even in the case of possibly singular varieties, for the multiplicativity $\chi_y(E) = \chi_y(F)\chi_y(B)$ to hold, the smoothness of the base variety $B$ is required and it is not known whether this smoothness can be dropped or not.  In this paper we consider only smooth varieties, thus we do not deal with this more general $\chi_y$-genus $\chi_y(X)= \int_X {T_y}_*(X)$, mainly because in this paper the duality formula (see below) of the Hirezebruch $\chi_y$-genus is crucial. 

In this paper we consider the difference of the Hirzebruch $\chi_y$-genus, $\chi_y(E) - \chi_y(F)\chi_y(B)$, more precisely we try to express the difference in terms of the Euler-Poincar\'e characteristic, the Todd genus and the signature. As a byproduct, it turns out that from the explicit formula describing the difference $\chi_y(E) - \chi_y(F)\chi_y(B)$ we obtain various results, for example we obtain 
that the signature of such a fiber bundle is multiplicative mod $4$, i.e. $\sigma(E) \equiv \sigma(F)\sigma(B) \op{mod} 4$, which is compatible with the above mentioned result of Hambleton-Korzeniewski-Ranicki \cite{HKR}. In this sense our result can be said to be \emph{an interesting application of the Hirzebruch $\chi_y$-genus to the ``multiplicativity problem" of the signature}, and also that only the Euler-Poincar\'e characteristic is multiplicative for any such fiber bundles.

In  the case of possibly singular varieties, for the difference of the $\chi_y$-genus $\chi_y(X)$ and the motivic Hirzebruch class $T_{y*}(X)$ and their intersection-homological analogues $I\chi_y(X)$ and $IT_{y*}(X)$, see the papers \cite{CS}, \cite{CMS}, \cite{CLMS1}, \cite{CLMS2}, \cite{MS} etc (cf. \cite{To}).

\section{Hirzebruch $\chi_y$-genus}

First we recall the  definition of the Hirzebruch $\chi_y$-genus. Let $X$ be a smooth complex algebraic 
variety. The
\emph{$\chi_{y}$-genus} of $X$ is defined by 
$$\chi_{y}(X):= 
\sum_{p\geq 0} \chi(X,\Lambda^{p}T^{*}X)y^p
= \sum_{p\geq 0} \left( \sum_{i\geq 0}
(-1)^{i}\op{dim}_{\mathbb C}H^{i}(X,\Lambda^{p}T^{*}X) \right)y^p\:.
$$

Thus the $\chi_y$-genus is the generating function of the Euler-Poincar\'e characteristic $\chi(X,\Lambda^{p}T^{*}X)$ of the sheaf $\Lambda^{p}T^{*}X$, which shall be simply denoted by $\chi^p(X)$:
$$\chi_{y}(X)= \sum_{p\geq 0} \chi^p(X)y^p.$$
Since $\Lambda^{p}T^{*}X =0$ for $p> \op{dim}_{\mathbb C}X$, $\chi_{y}(X)$ is a polynomial of at most degree $\op{dim}_{\mathbb C} X$. $\chi_y$ is multiplicative, i.e. $\chi_y(X \times Y) = \chi_y(X) \chi_y(Y)$. 


For the distinguished three values $-1, 0, 1$ of $y$, by the definition we have the following:

\begin{itemize}
\item the Euler-Poincar\'e characteristic:

 $\chi(X) = \chi_{-1}(X) = \chi^0(X)- \chi^1(X) + \chi^2(X)- \cdots +(-1)^n\chi^n(X),$

\item the Todd genus:

$\tau (X) = \chi_0(X) = \chi^0(X), \hspace{3cm}$

\item the signature:

$\sigma(X) = \chi_1(X) = \chi^0(X)+ \chi^1(X) + \chi^2(X)+\cdots +\chi^n(X).$
\end{itemize}

We also remark that it follows from \cite{Hir} that (duality formula)
\begin{equation}\label{duality}
\chi^p(X) = (-1)^n \chi^{n-p}(X),
\end{equation}
which plays an important role in this paper as we will see soon.
\section{HIrzebruch $\chi_y$-genera of fiber bundles}
\subsection{In the case of fiber bundle whose total complex dimension is odd}
\begin{thm} For a smooth compact complex algebraic variety of odd dimension $2n+1$, we have
\begin{align*}
\chi_y(X) & = \tau(X) \Bigl(1 +(-1)^{n+1}y^n \Bigr)  \Bigl(1 -(-1)^{n+1}y^{n+1} \Bigr) + (-1)^n \frac{\chi(X)}{2}y^n(1-y) +\\
& \hspace{5cm} \sum_{i=1}^{n-1} \chi^i(X) y^i\Bigl(1 - (-1)^{n-i}y^{n-i} \Bigr) \Bigl( 1 + (-1)^{n-i}y^{n-i+1}\Bigr) .
\end{align*}
\end{thm}
It is a tedious computation, but for the sake of the reader we write down a proof.
\begin{proof} By the duality formula $\chi^i(X) = (-1)^{2n+1}\chi^{2n+1-i}(X) = -\chi^{2n+1-i}(X)$, we get
\begin{align*}
\chi_y(X) & = \tau(X) + \sum_{i=1}^n\chi^i(X)y^i + \sum_{i=n+1}^{2n+1}(-\chi^{2n+1-i}(X)) y^i\\
& = \tau(X) + \sum_{i=1}^n\chi^i(X)y^i -\sum_{j=1}^{n+1} \chi^{n+1-j}(X)y^{n+j}\\
& = \tau(X)(1 - y^{2n+1}) + \sum_{i=1}^n\chi^i(X)(y^i - y^{2n+1-i}).
\end{align*}
Then we have
\begin{equation}\label{chi}
\chi(X) = \chi_{-1}(X) = 2\tau(X) + \sum_{i=1}^{n} (-1)^i2\chi^i(X).
\end{equation}
Hence we have
$$(-1)^{n+1}2\chi^n(X) = 2\tau(X) - \chi(X) + \sum_{i=1}^{n-1} (-1)^i2\chi^i(X).$$
namely, we have
$$\chi^n(X) = (-1)^{n+1}\tau(X) + (-1)^n \frac{\chi(X)}{2} + \sum_{i=1}^{n-1} (-1)^{n-i-1}\chi^i(X).$$
Thus we get 
\begin{align*}
\chi_y(X) & = \tau(X)(1 - y^{2n+1}) + \sum_{i=1}^{n-1}\chi^i(X)(y^i - y^{2n+1-i}) + \\
& \hspace{2.6cm}\Bigl ((-1)^{n+1}\tau(X) + (-1)^n \frac{\chi(X)}{2} + \sum_{i=1}^{n-1} (-1)^{n-i-1}\chi^i(X) \Bigr) (y^n - y^{n+1})\\
& = \tau(X)\Bigl (1 - y^{2n+1} + (-1)^{n+1}y^n- (-1)^{n+1}y^{n+1} \Bigr ) +\\
& (-1)^n \frac{\chi(X)}{2}y^n(1-y) + \sum_{i=1}^{n-1}\chi^i(X) y^i \Bigl (1 - y^{2n+1-2i} -(-1)^{n-i}y^{n-i} - (-1)^{n-i+1}y^{n-i+1} \Bigr).
\end{align*}
From which we get the above formula in the theorem.
\end{proof}

Since the Hirzebruch $\chi_y$-genus $\chi_y(X)$ is an integral polynomial,i.e.$\chi_y(X) \in \mathbb Z[y]$, the appearance of $\frac{\chi(X)}{2}$ in the above formula or the above equation (\ref{chi}) implies the following corollary, which one could obtain directly without using the Hirzebruch $\chi_y$-genus, because it seems to be a well-known fact in topology of manifolds that the Euler-Poincar\'e characteristic of a compact orientable smooth manifold $M$ of real dimension $4k+2$ is always even (which follows from the Poincar\'e duality and $\op{dim}_{\mathbb R}H^{2k+1}(M, \mathbb R)$ being even): 
\begin{cor} Let $X$ be a smooth compact complex algebraic varieties of odd complex dimension. Then its Euler-Poincar\'e characteristic $\chi(X)$ is even.
\end{cor}

\begin{cor} Let $E$ be a fiber $F$ bundle over a base $B$ such that $\op{dim}F + \op{dim}B = 2n+1$.
\begin{align*}
\chi_y(E)  -  & \chi_y(F) \chi_y(B) = \Bigl(\tau(E) - \tau(F) \tau(B) \Bigr) \Bigl(1 +(-1)^{n+1}y^n \Bigr)  \Bigl(1 -(-1)^{n+1}y^{n+1} \Bigr)  + \\
& \hspace{3cm}  \sum_{i=1}^{n-1} \Bigl (\chi^i(E) - \chi^i(F \times B) \Bigr) \Bigl(1 - (-1)^{n-i}y^{n-i} \Bigr) \Bigl( 1 + (-1)^{n-i}y^{n-i+1}\Bigr) .
\end{align*}
\end{cor}

\begin{rem} In the above corollary, we note that
$$\chi^i(F \times B) = \sum_{j=0}^i\chi^j(F)\chi^{i-j}(B),$$
which follows from the multiplicativity $\chi_y(F \times B) = \chi_y(F)\chi_y(B)$.
\end{rem}

\begin{cor} In the case when $n=0$ in the above theorem, and let $g(X)$ be the genus of the curve $X$. we have
$$\tau(X) = \frac{\chi(X)}{2} = 1-g(X)  \quad \text{and} \quad \chi_y(X) = \tau(X)(1-y) = \frac{\chi(X)}{2}(1-y) = \left(1-g(X) \right )(1-y).$$
\end{cor}

\begin{cor} In the case when $n=1$ in the above theorem, we have
$$\chi^1(X) = \tau(X) - \frac{\chi(X)}{2} \quad \text{and} \quad \chi_y(X) = \tau(X)(1+y)^2(1-y) -\frac{\chi(X)}{2}y(1-y).$$
\end{cor}

\begin{cor} Let $E$ be a fiber bundle with fiber $F$ and base $B$ such that $\op{dim} F=1$ and $\op{dim} B=2$ or $\op{dim} F=2$ and $\op{dim} B=1$, i.e., $E$ is a curve bundle over a surface or a surface bundle over a curve. 
\begin{enumerate}
\item The difference of the Hirzebruch $\chi_y$-genus is the following:
$$\chi_y(E) - \chi_y(F) \chi_y(B) = \Bigl (\tau(E) -\tau(F) \tau (B) \Bigr) (1+y)^2(1-y).$$
\item $\chi_y(E) =\chi_y(F) \chi_y(B)  \, \, (\forall y) \, \, \Longleftrightarrow \tau(E) =\tau(F)\tau (B).$
\end{enumerate}
\end{cor}


\begin{cor} In the case when $n=2$ in the above theorem, we have
$$\chi_y(X) = \tau(X)(1-y^2)(1+y^3) + \chi^1(X)y(1+y)^2(1-y)+ \frac{\chi(X)}{2} y^2(1-y).$$
\end{cor}


\begin{cor} Let $E$ be a fiber bundle with fiber $F$ and base $B$ such that $\op{dim} F=1$ and $\op{dim} B=4$, $\op{dim} F=2$ and $\op{dim} B=3$, $\op{dim} F=3$ and $\op{dim} B=2$, or $\op{dim} F=4$ and $\op{dim} B=1$, i.e., $E$ is a curve bundle over a $4$-fold, a surface bundle over a $3$-fold, a $3$-fold bundle over a surface, or a $4$-fold bundle over a curve. 
\begin{enumerate}
\item The difference of the Hirzebruch $\chi_y$-genus is the following:
\begin{align*}
\chi_y(E) - \chi_y(F)\chi_y(B) & = \Bigl (\tau(E) -\tau(F)\tau (B) \Bigr )(1-y^2)(1+y^3) + \\
& \hspace{1.3cm}\Bigl (\chi^1(E) - \tau(F)\chi^1(B) - \chi^1(F)  \tau(B) \Bigr) y(1+y)^2(1-y) .
\end{align*}
\item $\chi_y(E) =\chi_y(F)\chi_y(B) (\forall y)  \Leftrightarrow \tau(E) =\tau(F)\tau (B) \, \text{and} \,  \chi^1(E) = \tau(F)\chi^1(B) + \chi^1(F)  \tau(B).$
\end{enumerate}
\end{cor} 


\subsection{In the case of fiber bundles whose total complex dimension is even}

Here we consider the case of $\op{dim}_{\mathbb C}X =2n$.

By the duality formula (\ref{duality}) at the end of \S2, we have

\begin{align}\label{eq0}
\chi_y(X) & = \sum_{i=0}^{n-1}\chi^i(X) (y^i + y^{2n-i}) + \chi^n(X)y^n \\
& =  \sum_{i=0}^{n-1}\chi^i(X) y^i \Bigl (1+ y^{2n-2i} \Bigr) + \chi^n(X)y^n \label{eq0}
\end{align}
Thus we have
$$\qquad \qquad \quad \, \, \chi(X) = \chi_{-1}(X) = \sum_{i=0}^{n-1}(-1)^i2\chi^i(X)+ (-1)^n \chi^n(X),$$
$$\sigma(X) =\chi_1(X) = \sum_{i=0}^{n-1} 2\chi^i(X) + \chi^n(X).$$
Therefore we have
\begin{equation}\label{eq1}
\sigma(X) + \chi(X) = \sum_{i=0}^{n-1} 2\Bigl(1 + (-1)^i \Bigr) \chi^i(X)+ \Bigl (1 + (-1)^n \Bigr) \chi^n(X),
\end{equation}
\begin{equation}\label{eq2}
\, \, \, \, \, \, \, \, \, \, \, \sigma(X) - \chi(X) = \sum_{i=0}^{n-1} 2\Bigl(1 + (-1)^{i+1} \Bigr) \chi^i(X)+ \Bigl (1 + (-1)^{n+1} \Bigr) \chi^n(X).
\end{equation}
In order to continue further, we need to consider the case when $n=2k$ is even and the case when $n=2k+1$ is odd, i.e., $2n=4k, 4k+2$.

\begin{thm}
In the case when $\op{dim}_{\mathbb C}X = 4k$
\begin{align*}
\chi_y(X) = & \tau(X) (1-y^{2k})^2 + \frac{\sigma(X)}{4} y^{2k-1}(1+y)^2 - \frac{\chi(X)}{4} y^{2k-1}(1-y)^ 2 +\\
& \hspace{3.5cm} \sum_{j=1}^{k-1}\chi^{2j}(X) y^{2j} (1 - y^{2k-2j})^2 +\\
& \hspace{4cm}  \sum_{j=1}^{k-1}\chi^{2j-1}(X) y^{2j-1} (1-y^{2k-2j})(1-y^{2k-2j+2}).
\end{align*}
\end{thm}
\begin{proof} It follows from (\ref{eq1}) and (\ref{eq2}) that we have
\begin{equation}\label{eq3.-1}
\sigma(X) + \chi(X) = 4\tau(X) + 4\sum_{j=1}^{k-1}\chi^{2j}(X)+ 2\chi^{2k}(X),
\end{equation}
\begin{equation}\label{eq3.0}
\sigma(X) - \chi(X) = 4\sum_{j=1}^k \chi^{2j-1}(X).\hspace{3cm}
\end{equation}
From which we get the following:
\begin{equation}\label{eq3}
\chi^{2k}(X) = \frac{\sigma(X) + \chi(X)}{2} - 2\tau(X) - 2\sum_{j=1}^{k-1}\chi^{2j}(X),
\end{equation}
\begin{equation}\label{eq4}
\chi^{2k-1}(X) = \frac{\sigma(X) - \chi(X)}{4} - \sum_{j=1}^{k-1}\chi^{2j-1}(X). \hspace{1.5cm}
\end{equation}

Since $n=2k$, by singling out the last two terms of degree $2k-1$ and $2k$ of (\ref{eq0}), we have
\begin{equation}\label{eq5}
\chi_y(X) = \sum_{i=0}^{2k-2}\chi^i(X)y^i \Bigl(1 + y^{4k-2i} \Bigr) + \chi^{2k-1}(X)y^{2k-1}(1+y^2) + \chi^{2k}(X)y^{2k}.
\end{equation}
By plugging (\ref{eq3}) and (\ref{eq4}) into the above formula (\ref{eq5}), we get
\begin{align*}
\chi_y(X) & = \tau(X)\Bigl(1 + y^{4k} \Bigr) +  \sum_{i=1}^{2k-2}\chi^i(X)y^i \Bigl(1 + y^{4k-2i} \Bigr)  + \\
& \hspace{3cm}\Bigl (\frac{\sigma(X) - \chi(X)}{4} - \sum_{j=1}^{k-1}\chi^{2j-1}(X) \Bigr) y^{2k-1}(1+y^2)  + \\
& \hspace{4.5cm}\Bigl(\frac{\sigma(X) + \chi(X)}{2} - 2\tau(X) - 2\sum_{j=1}^{k-1}\chi^{2j}(X) \Bigr) y^{2k}.
\end{align*}
Then, by cleaning them out with respect to $\tau(X), \sigma(X), \chi(X)$, the ``even part" $\chi^{2j}(X)$ and the ``odd part" $\chi^{2j-1}(X)$, we get the formula in the theorem.
\end{proof}

\begin{thm}
In the case when $\op{dim}_{\mathbb C}X = 4k+2$
\begin{align*}
\chi_y(X) = & \tau(X) (1-y^{2k})(1- y^{2k+2}) + \frac{\sigma(X)}{4} y^{2k}(1 +y)^2+
\frac{\chi(X)}{4} y^{2k}(1-y)^2 +\\
& \hspace{3cm}  \sum_{j=1}^{k-1}\chi^{2j}(X) y^{2j}(1- y^{2k-2j})(1-y^{2k-2j+2}) +\\
& \hspace{4.5cm} \sum_{j=1}^k\chi^{2j-1}(X) y^{2j-1} (1 - y^{2k-2j+2})^2.
\end{align*}
\end{thm}
\begin{proof}
Since $n=2k+1$, it follows from (\ref{eq1}) and (\ref{eq2}) that we have
\begin{equation}\label{eq6.-1}
\sigma(X) + \chi(X) = 4\tau(X) + 4\sum_{j=1}^{k}\chi^{2j}(X) \hspace{1cm} 
\end{equation}
\begin{equation}\label{eq6.0}
\sigma(X) - \chi(X) = 4\sum_{j=1}^k \chi^{2j-1}(X) + 2\chi^{2k+1}(X).
\end{equation}
From which we get the following:
\begin{equation}\label{eq6}
\chi^{2k}(X) = \frac{\sigma(X) + \chi(X)}{4} - \tau(X) - \sum_{j=1}^{k-1}\chi^{2j}(X), \hspace{1cm}
\end{equation}
\begin{equation}\label{eq7}
\chi^{2k+1}(X) = \frac{\sigma(X) - \chi(X)}{2} - 2\sum_{j=1}^k \chi^{2j-1}(X).\hspace{2 cm} 
\end{equation}

Since $n=2k+1$, by singling out the last two terms of degree $2k$ and $2k+1$ of (\ref{eq0}), we have
\begin{equation}\label{eq8}
\chi_y(X) = \sum_{i=0}^{2k-1}\chi^i(X)y^i \Bigl(1 + y^{4k+2-2i} \Bigr) + \chi^{2k}(X)y^{2k}(1+y^2) + \chi^{2k+1}(X)y^{2k+1}.
\end{equation}
By plugging (\ref{eq6}) and (\ref{eq7}) into the above formula (\ref{eq8}), we get
\begin{align*}
\chi_y(X) & = \tau(X)\Bigl(1 + y^{4k+2} \Bigr) +  \sum_{i=1}^{2k-2}\chi^i(X)y^i \Bigl(1 + y^{4k+2-2i} \Bigr)  + \\
& \hspace{3cm} \Bigl (\frac{\sigma(X) + \chi(X)}{4} - \tau(X) -\sum_{j=1}^{k-1}\chi^{2j}(X)  \Bigr) y^{2k}(1+y^2)  + \\
& \hspace{6cm} \Bigl(\frac{\sigma(X) - \chi(X)}{2} - 2\sum_{j=1}^k \chi^{2j-1}(X) \Bigr) y^{2k+1}.
\end{align*}
Then, by cleaning them out with respect to $\tau(X), \sigma(X), \chi(X)$, the ``even part" $\chi^{2j}(X)$ and the ``odd part" $\chi^{2j-1}(X)$, we get the formula in the theorem.

\end{proof}

\begin{cor} 
(1) If $E$ is a fiber $F$ bundle over a base $B$ such that $\op{dim}_{\mathbb C}F + \op{dim}_{\mathbb C}B = 4k$, then the difference of the Hirzebruch $\chi_y$-genus is 
\begin{align*}
\chi_y(E) -& \chi_y(F) \chi_y(B) = \Bigl (\tau(E)-\tau(F) \tau(B)\Bigr)  (1-y^{2k})^2  + \\
& \hspace{4cm}  \frac{\sigma(E) -\sigma(F) \sigma(B)}{4} y^{2k-1}(1+y)^2 +\\
& \hspace{2cm} \sum_{j=1}^{k-1}\Bigl(\chi^{2j}(E) - \chi^{2j}(F \times B) \Bigr) y^{2j} (1 - y^{2k-2j})^2 +\\
& \hspace{1.5cm}  \sum_{j=1}^{k-1}\Bigl(\chi^{2j-1}(E) - \chi^{2j-1}(F \times B) \Bigr) y^{2j-1} (1 - y^{2k-2j})(1-y^{2k-2j+2}).
\end{align*}

(2) If $E$ is a fiber $F$ bundle over a base $B$ such that $\op{dim}_{\mathbb C}F + \op{dim}_{\mathbb C}B = 4k+2$, then the difference of the Hirzebruch $\chi_y$-genus is 
\begin{align*}
\chi_y(E) -& \chi_y(F)\chi_y(B) = \Bigl (\tau(E)-\tau(F) \tau(B)\Bigr) (1-y^{2k})(1- y^{2k+2}) + \\
& \hspace{3cm} \frac{\sigma(E) -\sigma(F)\sigma(B)}{4} y^{2k}(1 +y)^2 +\\
& \hspace{1.5cm} \sum_{j=1}^{k-1}\Bigl(\chi^{2j}(E) - \chi^{2j}(F \times B) \Bigr) y^{2j}(1- y^{2k-2j})(1 - y^{2k-2j+2}) +\\
& \hspace{3cm}  \sum_{j=1}^{k-1}\Bigl(\chi^{2j-1}(E) - \chi^{2j-1}(F \times B) \Bigr) y^{2j-1} (1 - y^{2k-2j+2})^2.
\end{align*}
\end{cor}

Since the Hirzebruch $\chi_y$-genus is an integral polynomial, from the above formulae we get the following theorem, which is compatible with the main result of Hambleton-Korzeniewski-Ranicki \cite{HKR}:
\begin{thm} Let $E$ be a fiber $F$ bundle over a base $B$. Then the following always hold: 
$$\sigma(E) \equiv \sigma(F)\sigma(B) \, \, \op{mod} \, \, 4 .$$
\end{thm}

\begin{rem} In the case when $\op{dim}_{\mathbb C}E$ is odd, either $\op{dim}_{\mathbb C} F$ or $\op{dim}_{\mathbb C} B$ is odd, thus by the definition of the signature $\sigma(E) = \sigma(F) \sigma(B) =0$, thus the above formula is trivial. So the above theorem is in fact for the case when $\op{dim}_{\mathbb C}E$ is even.
\end{rem}

\begin{rem}
In the above we try to give explicit formulae for the Hirzebruch $\chi_y$-genera $\chi_y(X)$ and the difference $\chi_y(E) -\chi_y(F)\chi_y(B)$ for a fiber bundle $F \hookrightarrow  E \to B$, and as a byproduct we get the multiplicativity of the signature modulo $4$. In fact, as far as such congruent expressions are concerned, using the formulae such as (\ref{eq3.-1}), (\ref{eq3.0}), (\ref{eq6.-1}), (\ref{eq6.0}) used in the above proof, we can get the following general congruent expressions.
\begin{enumerate}
\item Let $X$ be a smooth compact complex algebraic variety of complex dimension $4k$. Then it follows from (\ref{eq3.0}) that 
$$\sigma(X) - \chi(X)  \equiv 0  \,\, \op{mod} 4, \,\, \text{i.e.} \,\, \sigma(X) \equiv \chi(X) \,\, \op{mod} 4.$$
If we use (\ref{eq3.-1}) we get a weaker one:
$$\sigma(X) + \chi(X)  \equiv 0 \,\, \op{mod} 2, \,\, \text{i.e.} \,\, \sigma(X) \equiv \chi(X) \,\, \op{mod} 2.$$

\item Let $X$ and $Y$ be both smooth compact complex algebraic varieties of complex dimension divisible by $4$, where we do not necessarily require that the dimensions are the same. Using (\ref{eq3.0})  we get
$$\sigma(X)-\chi(X) \equiv \sigma(Y) - \chi(Y) \,\, \op{mod} 4, \,\, \text{i.e.} \,\, \sigma(X) -\sigma(Y) \equiv \chi(X) -\chi(Y) \op{mod} 4.$$
Using (\ref{eq3.-1})  we get a weaker one:
$$\sigma(X)+\chi(X) \equiv \sigma(Y) +\chi(Y) \,\, \op{mod} 2, \,\, \text{i.e.} \,\, \sigma(X) -\sigma(Y) \equiv \chi(X) -\chi(Y) \op{mod} 2.$$
Hence, if $\chi(X) = \chi(Y)$, such as the case when $X$ is a fiber $F$ bundle over a base $B$ and $Y = F\times B$, then we obtain $\sigma(X) \equiv \sigma(Y) \op{mod} 4$ and a weaker one 
$\sigma(X) \equiv \sigma(Y) \op{mod} 2$.

\item Let $X$ and $Y$ be both smooth compact complex algebraic varieties of complex dimension $4k+2$ and $4j+2$, where we do not necessarily require $k=j$. Using (\ref{eq6.-1})  we get
$$\sigma(X)+\chi(X) \equiv \sigma(Y) +\chi(Y) \,\, \op{mod} 4, \,\, \text{i.e.} \,\, \sigma(X) -\sigma(Y) \equiv \chi(Y) -\chi(X) \op{mod} 4.$$
Using (\ref{eq6.0})  we get a weaker one:
$$\sigma(X)-\chi(X) \equiv \sigma(Y) -\chi(Y) \,\, \op{mod} 2, \,\, \text{i.e.} \,\, \sigma(X) -\sigma(Y) \equiv \chi(X) -\chi(Y) \op{mod} 2.$$
Hence, if $\chi(X) = \chi(Y)$, such as the case when $X$ is a fiber $F$ bundle over a base $B$ and $Y = F\times B$, then we obtain $\sigma(X) \equiv \sigma(Y) \op{mod} 4$ and a weaker one 
$\sigma(X) \equiv \sigma(Y) \op{mod} 2$.
\end{enumerate}
\end{rem}

\begin{cor} In the case when $2n=2$, i.e. $n=1$ in the above arguments, we have\\
\begin{enumerate}
\item $\displaystyle \tau(X) = \frac{\sigma(X) + \chi(X) }{4}, \quad \chi^1(X) = \frac{\sigma(X) -\chi(X) }{2}.$\\
\item 
$\displaystyle \chi_y(X) = \frac{\sigma(X)}{4} (1+y)^2 + \frac{\chi(X)}{4} (1-y)^2.$\\
\end{enumerate}
\end{cor}

\begin{cor}\label{curve} Let $E$ be a curve $F$ bundle over a curve $B$. 
\begin{enumerate}
\item The difference of the Hirzebruch $\chi_y$-genus is the following:
$$\chi_y(E) - \chi_y(F) \chi_y(B) = \frac{\sigma (E)}{4} (1+y)^2,$$
namely, we have
$$\chi_y(E) = \Bigl (1-g(F) \Bigr) \Bigl (1-g(B) \Bigr )(1-y)^2 + \frac{\sigma (E)}{4} (1+y)^2.$$
\item \label{curve2} In particular, the signature $\sigma(E)$ is always divisible by $4$.
\end{enumerate}
\end{cor}
\begin{rem} The result (\ref{curve2}) of Corollary \ref{curve}  is compatible with the result of H. Endo \cite{Endo} and W. Meyer \cite {Mey} that the signature of surface bundle over surface is always divisive by $4$, where surface is a surface of real dimension, not of complex dimension.
\end{rem}

\begin{rem} In \cite{BD} J. Bryan and R. Donagi proves the following

\noindent {\bf Theorem}
\emph{For any integers $g,n \geqq 2$, there exists a connected algebraic surface $X_{g,n}$ of signature $\sigma(X_{g,n}) = \frac{4}{3}g(g-1)(n^2-1)n^{2g-3}$ that admits two smooth fibrations $\pi_1: X_{g,n} \to C$ and $\pi_2: X_{g,n} \to \widetilde D$ with base and fiber genus
$(b_i, f_i)$ equal to}
$$\quad \, \, (b_1, f_1) = \Bigl (g, \, g(gn-1)n^{2g-2} + 1 \Bigr ) \,\, \text{and} $$
$$(b_2, f_2) = \Bigl (g(g-1)n^{2g-2} + 1, \, gn \Bigr).$$
\emph{respectively.}

Here we note that the signature $\sigma(X_{g,n}) = \frac{4}{3}g(g-1)(n^2-1)n^{2g-3}$  is indeed divisible by $4$,in fact divisible by $8$, because $(n^2-1)n^{2g-3} = (n^2-1)n\cdot n^{2g-4}= (n-1)n(n+1)n^{2g-4}$ is divisible by $6$ (since $(n-1)n(n+1)$ is divisible by $6$), thus the signature is divisible by $4 \times 2$. Furthermore it follows from the above corollary the $\chi_y$-genus of the surface $X_{g,n}$ (for both fibrations) is given by 
$$\chi_y(X_{g,n}) =g(gn-1)n^{2g-2}(g-1)(1-y)^2 + \frac{1}{3}g(g-1)(n^2-1)n^{2g-3}(1+y)^2.$$
From which we get that
$$\chi(X_{g,n}) =4g(g-1)(gn-1)n^{2g-2} \,\, \text{and} \,\, \tau(X_{g,n}) = \frac{1}{3}g(g-1)n^{2g-3}(3gn^2-3n+n^2-1).$$
\end{rem}

From the above corollary we get the following theorem.
\begin{thm} $\chi_y$-genus is multiplicative for any such fiber bundle $F \to E \to B$ as above (without any dimension requirement) if and only if $y= -1$, i.e. the only Euler-Poincar\'e characteristic $\chi$ is multiplicative for any such fiber bundle.
\end{thm}
\begin{proof} Let $y$ be a value such that $\chi_y$-genus is multiplicative for any such fiber bundle $F \to E \to B$. Then consider it for a curve bundle over a curve by Atiyah \cite{At} or the above examples by Bryan-Donagi, in which case the signature of the total space $E$ is non-zero. Hence
we have to have $0= \chi_y(E)-\chi(F)\chi(B) = \frac{\sigma(E)}{4}(1+y)^2$, which implies that $y=-1$.  
\end{proof}

\begin{rem}
$\chi_y(E) = \chi_y(F) \chi_y(B)$ if and only if $\sigma (E) = 0$, and $\sigma (E) = 0$ is equivalent to the multiplicativity $\tau(E) = \tau(F) \tau(B).$ The latter follows from 
$$4 \tau(E) = \sigma(E) + \chi(E) \, \, \text{and} \, \,  \chi(E)=\chi(F)\chi(B) =\Bigl(2\tau(F) \Bigr) \Bigl(2\tau(B) \Bigr) =4\tau(F)  \tau(B).$$
\end{rem}


\begin{cor} In the case when $2n=4$, i.e. $n=2$ in the above proposition, we have\\

(1) $\displaystyle \chi^1(X) = \frac{\sigma(X) -\chi(X) }{4}, \quad  \chi^2(X) = \frac{\sigma(X) + \chi(X) - 4\tau(X)}{2}.$ \\

(2)
$\displaystyle \chi_y(X)  = \tau(X)(1 -y^2)^2+ \frac{\sigma(X)}{4}y(1+y)^2- \frac{\chi(X)}{4} y(1-y)^2.$ \\

\end{cor}

\begin{cor} 
\begin{enumerate}
\item Let $E$ be a curve $F$ bundle over a $3$-fold $B$ or a $3$-fold bundle over a curve $B$. Then the difference of the Hirzebruch $\chi_y$-genus is the following:
$$\chi_y(E) - \chi_y(F) \chi_y(B) = \Bigl (\tau(E) -\tau(F) \tau (B) \Bigr ) (1-y^2)^2 + \frac{\sigma(E)}{4} y(1+y)^2.$$
\item
Let $E$ be a surface $F$ bundle over a surface $B$. Then the difference of the Hirzebruch $\chi_y$-genus is the following:
\begin{align*}
\chi_y(E) - \chi_y(F) \chi_y(B) & = \Bigl (\tau(E) -\tau(F) \tau (B) \Bigr ) (1-y^2)^2 + \\
& \hspace{4.5cm}\left ( \frac{\sigma(E) -\sigma(F) \sigma(B)}{4} \right) y(1+y)^2.
\end{align*}
\item Let the situation be as both cases. 
$$\chi_y(E) =\chi_y(F) \chi_y(B)  \, \, (\forall y) \, \, \Longleftrightarrow \tau(E) =\tau(F) \tau (B) \, \, \text{and} \,\, \sigma(E) = \sigma(F)  \sigma(B).$$

\end{enumerate}
\end{cor}


\section{Final remarks}
We note that as to the general $\chi_y$-genus $\chi_y(X)$ of possibly singular varieties, the value for which $\chi_y$ is multiplicative for any fiber bundle with both fiber and base being possibly singular  
has to be $-1$. $\chi_y$-genus is a special case of the Hodge-Deligne polynomial $\chi_{u,v}(X)$ (defined by using the mixed Hodge structure), i.e. $\chi_y(X) = \chi_{y, -1}(X)$. We do know that the Hodge-Deligne polynomial is also multiplicative $\chi_{u,v}(X \times Y) = \chi_{u,v}(X)\chi_{u,v}(Y)$. Is it true that the values $(u, v)$ for which the Hodge-Deligne polynomial $\chi_{u,v}(X)$ is multiplicative for any fiber bundle with both fiber and base being possibly singular has to be $(-1, -1)$?

To get the above various results the duality formula (\ref{duality}) is crucial. Therefore, even if $E, F, B$ are singular varieties, as long as the $\chi_y$-genus $\chi_y(X)=\int_XT_{y*}(X)$ satisfies the duality formula (such a variety could be called ``$\chi_y$-duality variety", mimicking the name of Poincar\'e duality space), we would get the same results as above. For example, a projective simplicial toric variety as discussed in \cite{MS2} (cf. \cite{MSS}, \cite{Sch}) is such a variety (pointed out by J. Sch\"urmann) \footnote{We will deal with more general cases of singular varieties and the intersection homology $\chi_y$-genus $I\chi_y(X)$ in a different paper, since we would like to stick to smooth varieties in this paper for the sake of simplicity.}.


C. Rovi \cite{Rov} proved that for a fiber bundle $F \to E \to B$ the signature is multiplicative mod $8$, $\sigma(E) \equiv \sigma(F)\sigma(B) (\op{mod} 8)$ if $\pi_1(B)$ acts trivially on $(H_{2m}(F)/\op{torsion}) \otimes \mathbb Z_4$ (cf. A. Korzeniewski \cite{Kor}).  As to the signature mod $8$, S. Morita \cite{Mor} has proved that for a closed oriented manifold $M^{4k}$ of real dimension $4k$,
$$\sigma(M) \equiv BK(H^{2k}(M, \mathbb Z_2), q) (\op{mod} 8),$$
where $BK(H^{2k}(M, \mathbb Z_2), q)$ is the $\mathbb Z_8$-valued Brown-Kervaire invariant \cite{Bro}. 
Using this Morita's formula we get the following formula for our situation of the complex algebraic fiber bundle $F^{2n} \to E^{2n+2m} \to B^{2m}$ (where we consider the complex dimensions $2n, 2m$):
$$\sigma(E) - \sigma(F) \sigma(B) \equiv BK(H^{2n+2m}(E, \mathbb Z_2), q) - BK(H^{2n+2m}(F\times B, \mathbb Z_2), q)\,  (\op{mod} 8).$$
In \cite{Rov} (also see its concise version \cite{Rov2}) C. Rovi has analyzed the Brown-Kervaire invariant  difference $BK(H^{2n+2m}(E, \mathbb Z_2), q) - BK(H^{2n+2m}(F\times B, \mathbb Z_2), q)$ in more details.

At the moment we do not know how to capture the parity of the integer $\frac{\sigma(E)-\sigma(F) \sigma(B)}{4}$ and how to relate this parity $\frac{\sigma(E)-\sigma(F)\sigma(B)}{4} \op{mod} 2$ to Rovi's Arf-Kervaire invariant \cite{Rov} (suggested by A. Ranicki).\\

\noindent
{\bf Acknowledgements:} \,The author would like to thank Sylvain Cappell for informing him of the reference \cite{Rov}. He also would like to thank Paolo Aluffi, Markus Banagl, Laurentiu Maxim, Andrew Ranicki, Carmen Rovi and J\"org Sch\"urmann for useful comments and suggestions for an earlier version of the paper and pointing out typos.


\end{document}